\documentclass[11pt,twoside]{amsart}

\usepackage[latin1]  {inputenc}%
\usepackage[T1]      {fontenc }%
\usepackage          {amsmath }%
\usepackage          {amsfonts}%
\usepackage          {amssymb }%
\usepackage          {amsthm  }%
\usepackage          {a4wide  }%
\usepackage          {url     }%
\usepackage          {tikz    }%
\usepackage[bookmarks=false,pdfborder={0 0 0.05}]{hyperref}
\usepackage[all]{xy}
\usepackage{bm}

\usetikzlibrary{arrows}

\usepackage{enumitem, amsmath, amsfonts, amssymb, amsthm,  wasysym, graphics, graphicx, xcolor, frcursive,comment,bbm}

\usepackage{etex}

\definecolor{darkblue}{rgb}{0.0,0,0.7} 
\definecolor{darkred}{rgb}{0.7,0,0} 

\usepackage{hyperref}
\usepackage[all]{xy}
\usepackage[T1]{fontenc}

\usepackage{array}

\usepackage{MnSymbol}

\usepackage{float}
\restylefloat{table}

\def\defn#1{{\sf #1}}

\newcommand{\edgedir}{\mathbin{\tikz [semithick, baseline=-0.2ex,-latex, ->] \draw [->] (0pt,0.4ex) -- (1em,0.4ex);}} 

\newcommand{\ZZ}{\mathbb Z}

\newcommand{\spanr}{\mathrm{span}_{\mathbb{R}}}

\DeclareMathOperator{\Red}{Red}

\DeclareMathOperator{\GL}{GL}

\DeclareMathOperator{\Z}{\mathbb{Z}}

\DeclareMathOperator{\pac}{Pc}

\newtheorem{theorem}{Theorem}[section]
\newtheorem{corollary}[theorem]{Corollary}
\newtheorem{Proposition}[theorem]{Proposition}
\newtheorem{Lemma}[theorem]{Lemma}

\theoremstyle{definition}

\newtheorem{example}[theorem]{Example}

\newtheorem{question}[theorem]{Question}

\title[The centralizer of a Coxeter element]{The centralizer of a Coxeter element}

\author[R.~Hollenbach]{Ruwen Hollenbach}
\address{Ruwen Hollenbach, Leibniz Universit\"at Hannover, Germany}
\email{hollenbach@math.uni-hannover.de}

\author[P.~Wegener]{Patrick Wegener}
\address{Patrick Wegener, Technische Universit\"at Kaiserslautern, Germany}
\email{wegener@mathematik.uni-kl.de}

\subjclass[2010]{Primary 20F55}

\keywords{Coxeter groups, Coxeter element}

\date{\today}

\begin{document}
\newcolumntype{C}[1]{>{\centering\arraybackslash}m{#1}}

\begin{abstract}
We prove that the centralizer of a Coxeter element in an irreducible Coxeter group is the cyclic group generated by that Coxeter element. 
\end{abstract}

\maketitle

\tableofcontents

\section{Introduction}\label{sec:intro}

A classical result in finite Coxeter groups states that the centralizer of a Coxeter element $c$ is the cyclic group generated by $c$ provided that $W$ is irreducible (see \cite[Proposition 30]{Car72}). In \cite{Blokhina} this was then proved for infinite Coxeter groups of finite rank where the Coxeter diagram is either a simply-laced tree or of affine type. More recently, it has been shown that this result also holds for well-generated complex reflection groups \cite[Theorem 1.9]{Bessis2}.

In this paper we prove that the same statement holds for arbitrary infinite, irreducible Coxeter groups of finite rank.

\begin{theorem} \label{thm:main}
Let $(W,S)$ be an infinite irreducible Coxeter system of finite rank and $c \in W$ a Coxeter element. Then $C_W(c)=\langle c \rangle$.
\end{theorem}
\noindent
It is worth noting that our proof does not depend on the results in \cite{Blokhina}. The primary tool we will use in the proof of this theorem are the outward roots introduced by Krammer in \cite{Daan1}

\medskip

The structure of the paper is as follows. In Section 2 we first recall some basic definitions and facts about Coxeter groups. We then state some results about essential and straight elements as well as maximal proper parabolic subgroups. Afterwards, in Section \ref{sec:Outward}, we study the set of outward roots for straight elements. All of these results will be crucial in our proof of Theorem \ref{thm:main}. This proof is finally carried out in Section \ref{sec:Proof}. In Section \ref{sec:Artin} we give a short outlook on a possible generalization of our main theorem to Artin groups.

\medskip
\textbf{Acknowledgments.} The authors thank Gunter Malle and Sophiane Yahiatene for helpful comments and corrections. They also thank Theodosios Douvropoulos and Jon McCammond for helpful comments regarding Section \ref{sec:Artin}.

\section{Coxeter groups}

\subsection{Generalities}
In this subsection we state some well known definitions and properties for Coxeter groups. For details and proofs we refer to \cite{Hu90}.

Recall that a \defn{Coxeter group} is a group $W$ given by a presentation
$$
W = \langle  S \mid (st)^{m_{st}} = 1 ~\forall s,t \in S \rangle,
$$
where $(m_{st})_{s,t \in S}$ is a symmetric $(|S| \times |S|)$-matrix with entries in $\ZZ_{\geq 1} \cup \{ \infty \}$. These entries have to satisfy $m_{ss}=1$ for all $s \in S$ and $m_{st} \geq 2$ for all $s \neq t$ in $S$. If $m_{st} = \infty$, then there is no relation for $st$ in the above presentation. The pair $(W,S)$ is called a \defn{Coxeter system} and $|S|$ is called the \defn{rank} of $(W,S)$. Further, if $|W|$ is finite the system is called \defn{finite} and otherwise it is called \defn{inifinite}. We assume all Coxeter systems in this paper to be of finite rank.

To each Coxeter system $(W,S)$ there is an associated labeled graph, called \defn{Coxeter diagram} and denoted by $\Gamma(W,S)$. Its vertex set is given by $S$ and there is an edge between distinct $s, t \in S$ labeled by $m_{st}$ if $m_{st} > 2$. The Coxeter system $(W,S)$ is called \defn{irreducible} if $\Gamma(W,S)$ is connected. 

Each $w \in W$ can be written as a product $w = s_{1} \cdots s_{k}$ with $s_{i} \in S$. The \defn{length} $\ell(w) = \ell_S(w)$ is defined to be the smallest $k$ for which such an expression exists. The expression $w = s_{1} \cdots s_{k}$ is called \defn{reduced} if $k = \ell(w)$.

\medskip
Let $(W,S)$ be a Coxeter system and let $V$ be a vector space over $\mathbb{R}$ with a basis $\Delta = \{e_s \, \mid \, s \in S\}$. We equip $V$ with a symmetric bilinear form $B$ by setting
\[ B(e_{s}, e_{t})= - \cos \frac{\pi}{m_{st}}\]
for all $s, t \in S$. This term is understood to be $-1$ if $m_{st}= \infty$. The group $W$ can be embedded into $\operatorname{GL}(V)$ via its natural representation (or Tits representation) $\sigma: W \to \GL(V)$ that sends $s \in S$ to the reflection
\[ \sigma_s: V \rightarrow V, ~v \mapsto v- 2 B(e_s, v) e_s. \]
We set $w(e_s) := \sigma(w)(e_s)$ and
$$
\Phi = \Phi(W,S):= \{ w(e_s) \mid w \in W,~s \in S \}.
$$
The set $\Phi$ is called the \defn{root system} for $(W,S)$ and we refer to $\Delta$ as the \defn{simple system} for $\Phi$. We call a root $\alpha= \sum_{s \in S} a_s e_s$
\defn{positive} and write $\alpha>0$ if $a_s \geq 0$ for all $s \in S$ and \defn{negative} if $a_s \leq 0$ for all $s \in S$. Let $\Phi^+$ be the set consisting of the positive roots. It turns out that
$\Phi$ decomposes into positive and negative roots, i.e. $\Phi= \Phi^+ \dot{\cup} -\Phi^+$. 

If $\alpha = w(e_s) \in \Phi$ for some $w \in W$ and $s \in S$, then $wsw^{-1}$ acts as a reflection on $V$. It sends $\alpha$ to $-\alpha$ and fixes pointwise the hyperplane orthogonal to $\alpha$. We set $s_{\alpha} = wsw^{-1}$ and call $T= \{ wsw^{-1}\mid w \in W, ~s \in S\}$ the set of \defn{reflections} for $(W,S)$.\\

The natural representation of $W$ yields a dual action of $W$ on $V^*$ defined by 
\begin{align} \label{equ:DualAction}
(wx)(v)=x(w^{-1}v) ~\text{for }w \in W, v\in V, x\in V^*.
\end{align}
\noindent
Let $C= \{x \in V^*  \mid  x(e_s) >0 ~\text{for all } s \in S\}$. Recall that the \defn{Tits cone} is defined as $U= \bigcup_{w \in W} w \overline{C}$ where $\overline{C}$ denotes the topological closure of $C$ in $V^*$. We denote the topological interior of $U$ in $V^*$ by $U^\circ$. \\

For each subset $I \subseteq S$ the subgroup $W_I = \langle I \rangle$ is called a \defn{standard parabolic subgroup} of $W$. A subgroup of the form $w W_I w^{-1}$ for some $w \in W$ and $I \subseteq S$ is called a \defn{parabolic subgroup}. Note that if $wW_I w^{-1}$ is a parabolic subgroup, then $(wW_I w^{-1}, w I w^{-1})$ is itself a Coxeter system. Furthermore, we say that $w I w^{-1}$ is of \defn{spherical type} if $w W_I w^{-1}$ is a finite Coxeter group. In this case we also call $wW_Iw^{-1}$ \defn{spherical}. If $I \subseteq S$ then we call 
$$
\Phi_I := \{ w(e_s) \mid w \in W_I,~s \in I \}
$$
the root system associated to $(W_I, I)$. The corresponding simple system is $\Delta_I := \{ e_s \mid s \in I \}$. Note that $\Phi_I = \Phi \cap \spanr (\Delta_I)$ \cite[Lemma 3.1]{Qi}.\\

Let $(W,S)$ be a Coxeter system with $S=\{ s_1, \ldots , s_n \}$. Recall that an element of the form $c = s_{\pi(1)} \cdots s_{\pi(n)}$ is called a \defn{standard Coxeter element} where $\pi$ is any permutation of $\{ 1, \ldots , n\}$. Any conjugate of a standard Coxeter element in $W$ is called a \defn{Coxeter element}. Moreover, an element is called a \defn{(standard) parabolic Coxeter element} if it is a (standard) Coxeter element in a parabolic subgroup.

\subsection{Reduced reflection factorizations and parabolic subgroups}

Let $(W,S)$ be a Coxeter system with set of reflections $T$. Since $S \subseteq T$, each $w \in W$ is a product of elements in $T$. We define 
$$
\ell_T(w):= \min \{ k \in \ZZ_{\geq 0} \mid w=t_1 \cdots t_k, ~t_i \in T \}
$$
and call $\ell_T(w)$ the \defn{reflection length} of $w$. If $w= t_1 \cdots t_k$ with $t_i \in T$, we call $(t_1, \ldots , t_k)$ a \defn{reflection factorization} for $w$. If in addition $k=\ell_T(w)$, we say that $(t_1, \ldots , t_k)$ is \defn{reduced}. We denote by $\Red_T(w)$ the set of all reduced reflection factorizations for $w$.

\begin{theorem}[{\cite[Theorem 1.4]{BDSW14}}] \label{thm:RedRef}
Let $(W,S)$ be a Coxeter system with set of reflections $T$ and let $W'$ be a parabolic subgroup of $W$. Then $T'= T \cap W'$ is the set of reflections for $W'$ and for each $w \in W'$ we have $\Red_T(w) = \Red_{T'}(w)$.
\end{theorem}

\begin{theorem}\label{thm:generation}
Let $(W,S)$ be a Coxeter system with set of reflections $T$, $W'$ a parabolic subgroup of $W$ and $w \in W'$ a Coxeter element (that is, $w$ is a parabolic Coxeter element of $W$). If $(t_1, \ldots , t_k), ~(r_1, \ldots , r_k) \in \Red_{W' \cap T}(w)$, then 
$$
W' = \langle t_1, \ldots , t_k \rangle = \langle r_1, \ldots , r_k \rangle.
$$
\end{theorem}

\begin{proof}
Since $W'$ is a parabolic subgroup, there exist $I= \{s_1, \ldots , s_k\} \subseteq S$ and $x \in W$ such that $W' = x W_I x^{-1}$. After possible renumbering we can assume that $w = s_1^x \cdots s_k^x$, where $s_i^x:=x s_i x^{-1}$. In particular we have $( s_1^x ,\ldots, s_k^x) \in \Red_{W' \cap T}(w)$ and $W' = \langle s_1^x ,\ldots, s_k^x \rangle$. Now, if $(t_1, \ldots , t_k) \in \Red_{W' \cap T}(w)$, then $( s_1^x ,\ldots, s_k^x)$ and $(t_1, \ldots , t_k)$ lie in the same orbit under the Hurwitz action by \cite[Theorem 1.3]{BDSW14}. It is easy to see from its definition that the Hurwitz action preserves the generated group, that is $\langle s_1^x ,\ldots, s_k^x \rangle = \langle t_1 ,\ldots, t_k \rangle$.
\end{proof}

\subsection{Essential and straight elements}
Coxeter elements are both essential and straight. These properties will be crucial for our proof of the main theorem. We recall the definitions and state the necessary properties in this subsection.

Let $(W,S)$ be a Coxeter system. An element $w \in W$ is called \defn{essential} if it does not lie in any proper parabolic subgroup. 

\begin{Proposition} \label{prop:EssentialInfinite}
Let $(W,S)$ be an infinite Coxeter system. Then every essential element in $W$ has infinite order.
\end{Proposition}

\begin{proof}
This follows directly from a result of Tits, stating that each finite subgroup of $W$ is contained in a spherical parabolic subgroup \cite[Corollary D.2.9]{Davis}.
\end{proof}

\begin{Proposition}[{\cite[Corollary 2.5]{Paris}}] \label{prop:PowersEssential}
Let $(W,S)$ be irreducible and infinite, $w \in W$ and $p \in \ZZ_{>0}$. Then $w$ is essential if and only if $w^p$ is essential.
\end{Proposition}

\medskip
For a subset $X \subseteq W$ we define $\pac(X)$ to be the \defn{parabolic closure} of $X$, that is, $\pac(X)$ is the smallest parabolic subgroup of $W$ containing $X$. This is well defined since parabolic subgroups are closed under taking intersections \cite[Proposition 1.1]{Qi}.

It is known that Coxeter elements (and therefore their powers) are essential \cite[Theorem 3.1]{Paris}. We will extend this fact to so-called (weak) quasi-Coxeter elements. An element $w \in W$ is called a \defn{weak quasi-Coxeter element} (resp. a \defn{quasi-Coxeter element}) if there exists a reduced reflection factorization $(t_1, \ldots, t_m) \in \Red_T(w)$ such that $W= \pac(\{ t_1, \ldots , t_m\})$ (resp. such that $W = \langle t_1, \ldots ,t_m \rangle$). Obviously, every Coxeter element is a quasi-Coxeter element.

\begin{Proposition}
Let $w \in W$ and $(t_1, \ldots , t_m) \in \Red_T(w)$. Then 
$$
\pac(\{ w \}) = \pac(\{ t_1, \ldots , t_m\}).
$$
\end{Proposition}

\begin{proof}
First observe that $t_1, \ldots , t_m \in \pac(\{ t_1, \ldots , t_m\})$, thus $w \in \pac(\{ w \}) \cap \pac(\{ t_1, \ldots , t_m\})$. This intersection is again a parabolic subgroup by \cite[Proposition 1.1]{Qi}. Since $\pac(\{ w \})$ is the minimal parabolic subgroup containing $w$, we conclude $\pac(\{ w \}) \subseteq \pac(\{ t_1, \ldots , t_m\})$.

It remains to show that $\pac(\{ t_1, \ldots , t_m\}) \subseteq \pac(\{ w \})$. Since $\pac(\{ w \})$ is a parabolic subgroup, its set of reflections is given by $T':= T \cap \pac(\{ w \})$. By Theorem \ref{thm:RedRef} we have $\Red_T(w) = \Red_{T'}(w)$, hence $t_1, \ldots , t_m \in T' \subseteq \pac(\{ w \})$. Therefore we have $\pac(\{ t_1, \ldots , t_m\}) \subseteq \pac(\{ w \})$.
\end{proof}

As a direct consequence of this proposition we obtain the following.
\begin{Proposition} \label{prop:QuasiIsEssential}
Weak quasi-Coxeter elements are essential.
\end{Proposition}

\medskip

An element $w \in W$ is called \defn{straight} if $\ell(w^m) = |m| \ell(w)$ for all $ m \in \ZZ$. 

\begin{theorem}[{Speyer, \cite[Theorem 1]{Speyer}}] \label{thm:Speyer}
Let $(W,S)$ be infinite and irreducible. Then standard Coxeter elements in $W$ are straight.
\end{theorem}

\subsection{Conjugacy and normalizers of standard parabolic subgroups}

The goal of this subsection is to provide a criterion to decide whether two proper parabolic subgroups of maximal rank are conjugate in $W$. Furthermore, we show that a proper parabolic subgroup of maximal rank is self-normalizing if $W$ is irreducible. Both results will be needed later in our proof of the main theorem.

The following is a well known fact about Coxeter groups. A proof can be found for instance in \cite[Lemma 3.2]{Qi}.
\begin{Lemma} \label{lem:Conjugacy}
Let $(W,S)$ be a Coxeter system and $I,J \subseteq S$, $w \in W$ such that $W_I = w W_J w^{-1}$. Then $|I|= |J|$, $w_0(\Delta_J) = \Delta_I$ and $I = w_0 J w_0^{-1}$ for some $w_0 \in w W_J$.
\end{Lemma}

The situation is especially easy for proper parabolic subgroups of maximal rank.

\begin{Lemma} \label{lem:ConjMaxParabolic}
Let $(W,S)$ be an irreducible, infinite Coxeter system of rank $n$ and $I, J \subseteq S$ with $|I|=|J|=n-1$. Then $W_I$ and $W_J$ are conjugate in $W$ if and only if $I = J$. 
\end{Lemma}

To prove this Lemma we will use a criterion of Krammer \cite{Daan1} to decide whether two standard parabolic subgroups are conjugate. This criterion is based on previous work by Deodhar \cite{Deodhar}. We will give a short introduction.

Let $(W,S)$ be a Coxeter system, $I \subseteq S$ and $s \in S \setminus I$. We will identify a subset $J \subseteq S$ with its full subgraph of $\Gamma(W,S)$. Let $K$ be the connected component of $I \cup \{ s\}$ containing $s$.

\begin{itemize}
\item If $K$ is of spherical type, we define $\nu(I,s):= w_{K \setminus \{s \}} w_K$, where $w_{K \setminus \{s \}}$ (resp. $w_K$) is the longest element in $W_{K \setminus \{s \}}$ (resp. in $W_K$);
\item if $K$ is not of spherical type, then $\nu(I,s)$ is not defined.
\end{itemize}
If $\nu(I,s)$ is defined, then we have $\nu(I,s)^{-1} \Delta_I = \Delta_J$ for $J = (I \cup \{ s\}) \setminus \{ t \}$ for some $ t \in K$. 

We define a directed graph $G = G(W,S)$ whose vertices are the subsets of $S$. For subsets $I, J \subseteq S$ there is an edge $I \stackrel{s}{\edgedir} J$ if $\nu(I,s)$ is defined \textbf{and} $\nu(I,s)^{-1} \Delta_I = \Delta_J$.

\begin{Proposition}[{\cite[Corollary 3.1.7]{Daan1}}] \label{prop:GraphConjParab}
Let $(W,S)$ be a Coxeter system and $I, J \subseteq S$. Then $W_I$ and $W_J$ are conjugate in $W$ if and only if $I$ and $J$ are in the same connected component of the graph $G(W,S)$.
\end{Proposition}

\begin{proof}[Proof of Lemma \ref{lem:ConjMaxParabolic}]
Let $I \subseteq S$ with $|I| = n-1$. If $s \in S \setminus I$, then $I \cup \{s \}= S$. In particular, since $(W,S)$ is irreducible, the connected component $K$ of $I \cup \{s \}$ containing $s$ is the whole of $S$. Therefore $\nu(I,s)$ is not defined and $I$ is an isolated vertex in the graph $G(W,S)$. By Proposition \ref{prop:GraphConjParab} the parabolic subgroup $W_I$ is conjugate to $W_J$ for some $J \subseteq S$ if and only if $I = J$. 
\end{proof}

The graph $G(W,S)$ also contains substantial information about the normalizers of the standard parabolic subgroups. The following is from \cite[Section 3.1]{Daan1}. We fix a subset $I \subset S$ and denote the connected component of $G(W,S)$ containing $I$ by $\mathcal{K}^\circ$. Let $\mathcal{T}$ be a spanning tree of $\mathcal{K}^\circ$. For $J \in \mathcal{T}$, let
\[\mu(J)=\nu(I_0,s_0) \cdots \nu(I_t,s_t),\]
where 
\[I=I_0 \stackrel{s_0}{\edgedir} I_1 \stackrel{s_1}{\edgedir} \dots \stackrel{s_t}{\edgedir} I_{t+1}=J \]
is the unique non-reversing path in $\mathcal{T}$ from $I$ to $J$. For any edge $e=J_1 \stackrel{s}{\edgedir} J_2$ in $\mathcal{K}^\circ$ with $J_1, J_2 \in \mathcal{T}$ we set $\lambda(e)=\mu(J_1) \nu(J_1,s) \mu(J_2)^{-1}$.

\begin{Proposition}[{\cite[Proposition 2.1]{Brink}, \cite[Corollary 3.1.5]{Daan1}}] \label{prop:Normalizer}
Let $(W,S)$ be a Coxeter system and let $I \subseteq S$. Then:
\begin{enumerate}[label=(\alph*)]
\item $N_W(W_I)$ is the semidirect product of $W_I$ by the group $N_I=\{w \in W \, \mid \, w \Delta_I=\Delta_I\}$.
\item $N_I$ is generated by the $\lambda(e)$ where $e$ is an edge of $\mathcal{K}^\circ$ that is not an edge of $\mathcal{T}$.
\end{enumerate}
\end{Proposition}

As an easy corollary we get the following.

\begin{corollary} \label{cor:normalizer}
Let $(W,S)$ be an irreducible, infinite Coxeter system of rank $n$ and let $I \subseteq S$ with $|I|=n-1$. Then $N_W(W_I)=W_I$.
\end{corollary}

\begin{proof}
We already saw in the proof of Lemma \ref{lem:ConjMaxParabolic} that $I$ is an isolated vertex. In particular, the connected component $\mathcal{K}^\circ$ of $G(W,S)$ containing $I$ consists only of $I$. Hence $N_I$ is trivial and $N_W(W_I)=W_I$ by Proposition \ref{prop:Normalizer}.
\end{proof}

\subsection{Outward roots} \label{sec:Outward}

Let $(W,S)$ be an infinite Coxeter system of finite rank and let $c \in W$ be a Coxeter element. By work of Krammer \cite{Daan1} the centralizer $C_W(c)$ of $c$ acts on the so-called outward roots for $c$. To better understand $C_W(c)$ and eventually prove Theorem \ref{thm:main}, we study this action in more detail in the remainder of this paper. We start by describing the outward roots for straight elements. \\

Let $w \in W$. We call a root $\alpha \in \Phi$ \defn{outward} for $w$ if the following holds for some $x \in U^\circ$. For almost all $m \in \mathbb{Z}$, we have $m  (w^m x (\alpha))  < 0$. The set of outward roots is denoted by $\operatorname{Out}(w)$. It is obvious that for every outward root $\alpha \in \operatorname{Out}(w)$ and every $k \in \mathbb{Z}$ we have $w^k \alpha \in \operatorname{Out}(w)$. In other words $\langle w \rangle$ acts on $\operatorname{Out}(w)$.  By \cite[Lemma 5.2.6]{Daan1}, the cardinality $r(w)$ of the set of orbits $\langle w \rangle \setminus \operatorname{Out}(w)$ for this action is finite. {Moreover, it follows from \cite[Corollary 5.2.4]{Daan1} that $C_{W}(w)$ acts on $\operatorname{Out}(w)$ (note the connection between \textit{odd} and \textit{outward} roots described in \cite[Definition 5.5.6]{Daan1}).

\begin{Proposition} \label{outward straight} If $w \in W$ is straight then $r(w)=\ell(w)$.
\end{Proposition}

\begin{proof}
By \cite[Corollary 5.6.6]{Daan1} we have $\lim\limits_{n \to \infty} \frac{\ell(w^n)}{n} = r(w)$. The assertion follows since $w$ is straight.
\end{proof}

For $w \in W$ we denote the set of positive roots that $w$ sends to negative roots by $\Phi^+(w)$. This set is finite and can be given explicitly. For simplicity we denote a simple root $e_{s_i}$ by $e_i$. If $w=s_{j_1} \cdots s_{j_k}$ ($s_{j_i} \in S)$ is a reduced expression then $\Phi^+(w)$ consists of the $k$ distinct roots  $s_{j_k} s_{j_{k-1}} \cdots s_{j_{i+1}}(e_{j_i})$ for $i \in \{1, \dots , k-1\}$ and $e_{i_k}$ (see for example \cite[Excercise II.5.6.1]{Hu90}). In particular, $\Phi^+(w^{-1})$ then consists of the $k$ distinct roots $\beta_{i}:=s_{j_1} \cdots s_{j_{i-1}}(e_{j_i})$ for $i \in \{2, \dots k\}$ and $\beta_{1}:=e_{j_1}$.

\begin{Proposition} \label{prop:outward roots} \label{prop:Representatives}
Let $w \in W$ be straight. Then $\Phi^+(w^{-1})$ is a set of representatives for $\langle w \rangle \setminus \operatorname{Out}(w)$.
\end{Proposition}

\begin{proof}
Since $w$ is straight, it follows that $\Phi^+(w^{-1}) \subseteq \Phi^+(w^{-m})$ for every $m \geq 1$. We will show that $\Phi^+(w^{-1}) \subseteq \Phi^+ \setminus \Phi^+(w^{m})$ for every $m \geq 1$. Let $w=s_{j_1} \cdots s_{j_k}$ be a reduced expression and let $\beta_{i}$ be as above. Then  
\begin{align*}
w^m s_{\beta_i}=(\underbrace{s_{j_1} \cdots s_{j_k} \cdots s_{j_1} \cdots s_{j_k}) (s_{j_1} \cdots s_{j_{i-1}} s_{j_i}}_{=:w'} s_{j_{i-1}} \cdots s_{j_1}).
\end{align*}
Since $w$ is straight the subword $w'$ is reduced. Hence $\ell(w')=mk + i$. Multiplying $w'$ with any $s \in S$ decreases the length by at most 1. It follows that 
\[\ell(w^m s_{\beta_i}) \geq mk+i - (i-1) > mk = \ell(w^m). \]
So we have $w^m \beta_i >0$ by \cite[Chapter II, Proposition 5.7]{Hu90}. Thus, $\Phi^+(w^{-1}) \subseteq \Phi^+ \setminus \Phi^+(w^{m})$.

Let $x \in C \cap U^\circ$. We observe that for every $\alpha \in \Phi$, $x(\alpha) > 0$ if and only if $\alpha >0$. If $m \geq 1$ by the above we therefore have
\begin{align*}
m(w^mx(\beta_i))& \stackrel{(\ref{equ:DualAction})}{=} m (\underbrace{x(w^{-m}\beta_i)}_{<0}) < 0, \text{ and} \\
-m (w^{-m}x(\beta_i)) & \stackrel{(\ref{equ:DualAction})}{=} -m (\underbrace{x(w^{m}\beta_i)}_{>0}) <0.
\end{align*}

Hence $\Phi^+(w^{-1}) \subseteq \operatorname{Out}(w)$ and it is clear that each orbit in $\langle w \rangle \setminus \operatorname{Out}(w)$ contains at most one element of $\Phi^+(w^{-1})$. By Proposition \ref{outward straight}, $\Phi^+(w^{-1})$ is therefore a set of representatives for $\langle w \rangle \setminus \operatorname{Out}(w)$.
\end{proof}

\section{The proof} \label{sec:Proof}
Throughout this section we assume that $(W,S)$ is an irreducible, infinite Coxeter system of rank $n$ and $S=\{s_1, \ldots , s_n\}$. \\

As stated in the last section, our proof of Theorem \ref{thm:main} relies on a detailed study of the action of $C_W(c)$ on $\operatorname{Out}(c)$. 
Furthermore, (for the Coxeter groups that are not affine) we will use the following result which is an immediate consequence of \cite[Corollary 6.3.10]{Daan1}. Note that every Coxeter element $c$ of $W$ is essential by Proposition \ref{prop:QuasiIsEssential}.

\begin{Proposition} \label{prop:ElInCenAreE}
Let $w \in W$ be an essential element. Then each element in $C_W(w)$ has either finite order or is essential.
\end{Proposition}

\begin{proof}
By Proposition \ref{prop:EssentialInfinite} an element of finite order can not be essential. The assertion is obviously true for affine Coxeter systems since each proper parabolic subgroup is finite in these cases. Therefore let $(W,S)$ be not affine. By \cite[Corollary 6.3.10]{Daan1}, the index of $\langle w \rangle$ in $C_W(w)$ is finite. Let $k$ be that index and let $v \in C_W(c)$ be of infinite order. Then $1 \neq v^k \in \langle w \rangle$. In particular, there exists $m \in \ZZ$ such that $v^k = w^m$. Since $w$ is essential, $w^m$ is essential by Proposition \ref{prop:PowersEssential}. That is, $v^k$ is essential. By applying Proposition \ref{prop:PowersEssential} again it then follows that $v$ is essential. 
\end{proof}

\begin{proof}[\textbf{Proof of Theorem \ref{thm:main}}]
Since $C_W(wcw^{-1})=wC_W(c)w^{-1}$ for every $w \in W$, we can assume our Coxeter element to be standard. After possible relabeling of our set $S$ we can therefore assume that $c=s_1 \cdots s_n$. Let $\beta_i=s_1 \cdots s_{i-1}(e_i)$ (with $\beta_1=e_{1}$) be as in Section 2.5. Note that 
\begin{align} \label{equ:Beta1}
s_{\beta_i}=s_1 \cdots s_{i-1} s_i s_{i-1} \cdots s_1 = c(s_n s_{n-1} \cdots s_{i+1} s_{i-1} \cdots s_1).
\end{align}
Let $g \in C_W(c)$. By Proposition \ref{prop:Representatives} and the fact that $C_W(c)$ acts on $\operatorname{Out}(c)$ there exist $m_i \in \Z$ and $j \in \{1, \dots, n\}$ such that $g \beta_i =c^{m_i} \beta_j$.
In other words
\begin{align} \label{equ:Beta2}
(gc^{-m_i})s_{\beta_i}(gc^{-m_i})^{-1} = s_{gc^{-m_i} (\beta_i)} = s_{\beta_j}.
\end{align}
We set $h_i:=gc^{-m_i}$. Clearly, $h_i \in C_W(c)$.
It follows from above that
\begin{align*}
h_i s_{\beta_i} h_i^{-1} & \stackrel{(\ref{equ:Beta1})}{=} h_i c (s_n s_{n-1} \dots s_{i+1} s_{i-1} \dots s_1) h_i^{-1} \\
& \stackrel{\phantom{(\ref{equ:Beta1})}}{=} c \left(  h_i  (s_n s_{n-1} \dots s_{i+1} s_{i-1} \dots s_1) h_i^{-1} \right)\\
& \stackrel{(\ref{equ:Beta2})}{=} c (s_n s_{n-1} \dots s_{j+1} s_{j-1} \dots s_1).
\end{align*}
Thus,
\[ h_i  (s_n s_{n-1} \dots s_{i+1} s_{i-1} \dots s_1) h_i^{-1}= s_n^{h_i} s_{n-1}^{h_i} \dots s_{i+1}^{h_i} s_{i-1}^{h_i} \dots s_1^{h_i}=s_n s_{n-1} \dots s_{j+1} s_{j-1} \dots s_1.  \]

Since $s_n s_{n-1} \dots s_{j+1} s_{j-1} \dots s_1$ is a (standard) parabolic Coxeter element we have by \cite[Lemma 2.1]{BDSW14} that
$$
n-1=\ell_T(s_n s_{n-1} \dots s_{j+1} s_{j-1} \dots s_1)=\ell(s_n s_{n-1} \dots s_{j+1} s_{j-1} \dots s_1).
$$
Therefore, $(s_n^{h_i}, s_{n-1}^{h_i}, \ldots , s_{i+1}^{h_i}, s_{i-1}^{h_i}, \ldots , s_1^{h_i} ) \in \Red_T(s_n s_{n-1} \cdots s_{j+1} s_{j-1} \cdots s_1)$. Let $I= \{1, \dots, n\} \setminus \{i\}$ and let $J=\{1, \dots, n\} \setminus \{j\}$. By Theorem \ref{thm:generation} it follows that 
\[h_iW_I h_i^{-1}=W_J. \]
Hence, by Lemma \ref{lem:ConjMaxParabolic} $I=J$, that is $h_i \in N_W(W_I)$.  However, $N_W(W_I)=W_I$ by Corollary \ref{cor:normalizer}. At this stage of the proof, we need to distinguish between affine and non-affine Coxeter groups.

First, suppose that $W$ is not affine. Then, since $h_i \in W_I$,  $h_i$ is not essential. So $h_i$ has finite order by Proposition \ref{prop:ElInCenAreE}. In particular, if $k$ denotes the index of $\langle c \rangle$ in $C_W(c)$ (which is finite by \cite[Corollary 6.3.10]{Daan1}) then $h_i^k=1$. Hence, for every $i \in \{1, \dots ,n\}$ we have
\[g^k=c^{m_ik}.\]
Since the order of $c$ is infinite it follows that all the $m_i$ are the same and we may set $m:=m_1$. Thus all the $h_i$ are the same and we may set $h:=h_1$. Note that 
\[\bigcap_{\substack{I \subseteq S \\ |I|=n-1}}I= \emptyset.\] 
In conclusion (for the first equality see \cite[Theorem 5.5]{Hu90}),

\[h \in \bigcap_{\substack{I \subseteq S \\ |I|=n-1}} W_I=W_\emptyset=\{1\}. \]
Thus $g=c^m$.

Now, suppose that $(W,S)$ is affine. In this case $W_I$ is a finite group for every $I \subseteq S$ with $|I|=n-1$. Let $l=\operatorname{lcm}(|W_I|  \mid I \subseteq S,~|I|=n-1 )$. Then $h_i^l=1$, that is $g^l=c^{m_i l}$ for every $i \in \{1, \dots ,n\}$.  Hence, as in the non-affine case, all the $m_i$ are the same. Let $m:=m_1$. Then, as before we conclude that $g=c^m$.
\end{proof}

\medskip
However, it is not the case that $C_W(w)= \langle w \rangle$ for an arbitrary essential element $w \in W$. The centralizer of an essential might not even be cyclic. In contrast to Coxeter elements, the centralizer of an arbitrary essential element can contain elements of finite order different from the identity (compare Proposition \ref{prop:ElInCenAreE}). We illustrate this in the following example.

\begin{example}
Consider a Coxeter system $(W,S)$ of type $\widetilde{D}_4$, that is, the Coxeter graph $\Gamma(W,S)$ is given as follows.
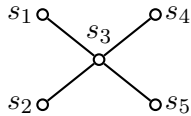
\begin{figure}[H]
\centering
\begin{tikzpicture}[scale=0.75, thick,>=latex]

  \node (1) at (2,-0.2) [circle, draw, inner sep=0pt, minimum width=4pt]{};
  \node (1b) at (1.6,-0.2) []{$s_1$};
  \node (3) at (3,-1) [circle, draw, inner sep=0pt, minimum width=4pt]{};
  \node (3b) at (3,-0.55) []{$s_3$};
  \node (2) at (2,-1.8) [circle, draw, inner sep=0pt, minimum width=4pt]{};
  \node (2b) at (1.6,-1.8) []{$s_2$};
  \node (4) at (4,-0.2) [circle, draw, inner sep=0pt, minimum width=4pt]{};
  \node (4b) at (4.4,-0.2) []{$s_4$};
  \node (5) at (4,-1.8) [circle, draw, inner sep=0pt, minimum width=4pt]{};
  \node (5b) at (4.4,-1.8) []{$s_5$};

  \draw[-] (1) to (3);
  \draw[-] (2) to (3);
  \draw[-] (3) to (4);
  \draw[-] (3) to (5);

%
\end{tikzpicture}
\end{figure}

The element $v=s_4s_3s_4s_5s_3s_2$ is a quasi-Coxeter element (but not a Coxeter element) in the spherical parabolic subgroup $W' =\langle s_2, s_3, s_4, s_5 \rangle$. Direct calculations yield $s_5s_4 \in W' \setminus \langle v \rangle$ as well as $s_5s_4 \in C_{W'}(v)$. The element $w= v s_1 \in W$ is quasi-Coxeter, hence essential by Proposition \ref{prop:QuasiIsEssential}. In particular, since $s_1$ commutes with $s_4$ and $s_5$, we have $s_5s_4 \in C_W(w)$. But $s_5s_4$ is not essential, hence $s_5s_4 \notin \langle w \rangle$ by Proposition \ref{prop:PowersEssential}.
\end{example}

\medskip

\section{Outlook: Artin groups} \label{sec:Artin}
Let $(W,S)$ be a Coxeter system with $S = \{s_1, \ldots , s_n \}$. The \defn{Artin group} associated to $(W,S)$ is the group given by the presentation
$$
A(W,S) =  \langle ~\bm{s_1}, \ldots , \bm{s_n} \mid \underbrace{\bm{ s_is_js_i} \cdots }_{m_{ij}~\text{terms}} = \underbrace{\bm{s_js_is_j} \cdots }_{m_{ij}~\text{terms}} ~\text{for all }i \neq j ~\rangle.
$$
Although closely related to Coxeter groups, these groups are rather mysterious and not much is known about them in general. 

Similar to our study of the centralizer of Coxeter elements in Coxeter groups we can consider the element $\bm{c} = \bm{s_1s_2} \cdots \bm{s_n} \in A(W,S)$ and try to determine its centralizer in $A(W,S)$. Clearly the center of $A(W,S)$ is contained in every centralizer in $A(W,S)$. Therefore we first might want to know what the center of $A(W,S)$ looks like. In fact, it is trivial in most cases if $(W,S)$ is infinite and irreducible \cite{Charney}.
If $(W,S)$ is irreducible and finite, the center of $A(W,S)$ is infinite cyclic and either generated by the element $(\bm{s_1s_2} \cdots \bm{ s_n})^{h}$ or the element $(\bm{s_1s_2} \cdots \bm{s_n})^{h/2}$ \cite[Satz 7.2]{Brieskorn}, where $h$ denotes the Coxeter number.

We have a natural homomorphism
$$
p: A(W,S) \rightarrow W, ~\bm{s_i} \mapsto s_i ~(1  \leq i \leq n).
$$
Given a reduced expression $w=s_{i_1} \dots s_{i_k} \in W$, we call $\bm{ w}=\bm{s_{i_1}} \cdots \bm{s_{i_k}} \in A(W,S)$ the \defn{lift} of this expression. In particular, we have $p(\bm{w})=w$ and the element $\bm{ s_1s_2} \cdots \bm{s_n} \in A(W,S)$ is the lift of the Coxeter element $c=s_1 s_2 \cdots s_n$ in $W$.

As an easy consequence of Theorem \ref{thm:main} we obtain the following.
\begin{corollary} \label{cor:main}
Let $(W,S)$ be an irreducible Coxeter system. If $\bm{c} = \bm{s_1s_2} \cdots \bm{s_n} \in A(W,S)$ is the lift of the standard Coxeter element $c = s_1 s_2 \cdots s_n \in W$, then $p(C_{A(W,S)}(\bm{c}))  \subseteq \langle c \rangle$. Furthermore, if $(W,S)$ is infinite and $C_{A(W,S)}(\bm{c})$ is cyclic, then $C_{A(W,S)}(\bm{c}) = \langle \bm{c} \rangle$.
\end{corollary}

\begin{proof}
The first assertion follows directly from Theorem \ref{thm:main}. Therefore let $(W,S)$ be infinite and assume $C_{A(W,S)}(\bm{c})$ to be cyclic with generator $x$. We have to show that $x = \bm{c}^{\pm 1}$. Since $\bm{c} \in C_{A(W,S)}(\bm{c})$, there exists $k \in \ZZ$ such that $x^k = \bm{c}$. By the first assertion we have $p(x) = c^m$ for some $m \in \ZZ$. Hence we obtain
$$
c^{mk} = (c^m)^k = p(x)^k = p(x^k) =p(\bm{c}) = c,
$$
that is, $mk =1$ as $c \in W$ has infinite order. Since $m,k \in \ZZ$, it follows that $m=k= \pm 1$. Thus, $x = \bm{c}^{\pm 1}$ as desired
\end{proof}

As a consequence of the work of Bessis \cite{Bessis2} as well as McCammond--Sulway \cite{McCammond} we obtain the follwing.

\begin{corollary}
If $(W,S)$ is an irreducible Coxeter system which is finite or affine, then $C_{A(W,S)}(\bm{c}) = \langle \bm{c} \rangle$.
\end{corollary}

\begin{proof}
For the case that $(W,S)$ is affine, the assertion is already proved implicitly in the proof of \cite[Proposition 11.9]{McCammond}. Now suppose that $(W,S)$ is finite and let $V$ be the complexification of its natural representation. For $n \in \mathbb{N}$, we let $\xi_n$ denote a primitive $n$-th root of unity. Then $p(\bm{c}) = c$ is a $\xi_h$-regular element (for instance, see \cite[Definiton 1.8]{Bessis2} for the definition of $\xi_h$-regular elements) where $h$ is the Coxeter number of $(W,S)$. Let $V' : =\ker(c-\xi_h)$ and interpret $\langle c \rangle$ as a complex reflection group acting on $V'$. Further, let $V'_{\text{reg}}$ be the associated hyperplane complement. By \cite[Theorem 12.4]{Bessis2}, $C_{A(W,S)}(\bm{c}) \cong \pi_1(\langle c \rangle \setminus V'_{\text{reg}})$. By \cite[Theorem 4.2]{Springer}, $V'$ is of dimension one and thus $V'_{\text{reg}}$ is homeomorphic to $\mathbb{R}^2 \setminus \{0\}$. Hence, $\pi_1(\langle c \rangle \setminus V'_{\text{reg}}) \cong \mathbb{Z}$. The assertion follows by Corollary \ref{cor:main}.
\end{proof}

In view of the previous statements we want to pose the following question:

\begin{question}
Does the conclusion of Theorem \ref{thm:main} hold for Artin groups? More precisely, if $(W,S)$ is an irreducible Coxeter system and $\bm{c}= \bm{ s_1s_2} \cdots \bm{s_n} \in A(W,S)$ is the lift of a Coxeter element in $W$, is it true that $C_{A(W,S)}(\bm{c}) = \langle \bm{c} \rangle$?
\end{question}

\nocite{*}

\bibliography{mybibCentralizer}
\bibliographystyle{amsplain}

\end{document}